\documentclass[reqno,11pt]{amsart}
\usepackage{amsmath,amssymb,amsfonts}

\usepackage{xcolor}
\colorlet{darkgreen}{green!50!black}

\usepackage{tikz}
\usetikzlibrary{matrix,arrows}

\usepackage[pagebackref,colorlinks,citecolor=red,linkcolor=blue]{hyperref}

\renewcommand*{\backref}[1]{}
\renewcommand*{\backrefalt}[4]{%
    \scriptsize%
    {
    \ifcase #1 (\textcolor{red}{Uncited.})%
          \or (Cit.\ on p.~#2)%
          \else (Cit.\ on pp.~#2)%
    \fi%
    }
}

\usepackage{fouriernc}
\usepackage[applemac]{inputenc}

\title[Derived category of a geometric stack]{On the derived category of quasi-coherent sheaves on an Adams geometric stack}

\author[L. Alonso]{Leovigildo Alonso Tarr\'{\i}o}
\address[L. A. T.]{Departamento de Matem\'a\-ticas\\
Universidade de Santiago de Compostela\\
E-15782  Santiago de Compostela, Spain}
\email{leo.alonso@usc.es}

\author[A. Jerem\'{\i}as]{Ana Jerem\'{\i}as L\'opez}
\address[A. J. L.]{Departamento de Matem\'a\-ticas\\
Universidade de Santiago de Compostela\\
E-15782  Santiago de Compostela, Spain}
\email{ana.jeremias@usc.es}

\author[M. P\'erez]{Marta P\'erez Rodr\'{\i}guez}
\address[M. P. R.]{Departamento de Matem\'a\-ticas\\
Esc. Sup. de Enx. Inform\'atica,
Campus de Ourense\\
Universidade de Vigo\\
E-32004 Ourense, Spain}
\email{martapr@uvigo.es}

\author[M. J. Vale]{Mar\'{\i}a J. Vale Gonsalves}
\address[M. J. V.]{Departamento de Matem\'a\-ticas\\
Universidade de Santiago de Compostela\\
E-15782  Santiago de Compostela, Spain}
\email{mj.vale@usc.es}

\thanks{This work has been partially supported by
Spain's MINECO and E.U.'s FEDER research projects MTM2011-26088, MTM2014-59456
and Xunta de Galicia's GRC2013-045}


\date{April, 26, 2017, \emph{typeset}: \today}

\theoremstyle{plain}
\newtheorem{thm}{Theorem}[section]
\newtheorem{lem}[thm]{Lemma}
\newtheorem{cor}[thm]{Corollary}
\newtheorem{prop}[thm]{Proposition}

\theoremstyle{remark}
\newtheorem*{rem}{Remark}

\theoremstyle{definition}

\newtheorem*{ex}{Example}
\newtheorem*{ack}{Acknowledgements}
\newtheorem{cosa}[thm]{}

\numberwithin{equation}{thm}

\newcommand{\CA}{\mathcal{A}}

\newcommand{\CE}{{\mathcal E}}
\newcommand{\CF}{{\mathcal F}}
\newcommand{\CG}{{\mathcal G}}
\newcommand{\CH}{{\mathcal H}}
\newcommand{\CI}{{\mathcal I}}

\newcommand{\CO}{\mathcal{O}}
\newcommand{\CP}{\mathcal{P}}

\newcommand{\SB}{\mathsf{B}}

\newcommand{\SF}{\mathsf{F}}

\newcommand{\SP}{\mathsf{P}}
\newcommand{\SQ}{\mathsf{Q}}
\newcommand{\SR}{\mathsf{R}}
\renewcommand{\SS}{\mathsf{S}}
\newcommand{\SST}{\mathsf{T}}
\newcommand{\ST}{\boldsymbol{\mathsf{T}}}
\newcommand{\SU}{\mathsf{U}}

\newcommand{\SW}{\mathsf{W}}
\newcommand{\SX}{\boldsymbol{\mathsf{X}}}

\newcommand{\CCC}{\boldsymbol{\mathsf{C}}}
\newcommand{\D}{\boldsymbol{\mathsf{D}}}
\newcommand{\K}{\boldsymbol{\mathsf{K}}}
\newcommand{\LL}{\boldsymbol{\mathsf{L}}}
\newcommand{\R}{\boldsymbol{\mathsf{R}}}
\newcommand{\A}{\mathsf{A}}

\newcommand{\sch}{\mathsf {Sch}}

\newcommand{\cc}{\mathsf{c}}

\newcommand{\md}{\text{-}\mathsf{Mod}}
\newcommand{\com}{\text{-}\mathsf{coMod}}

\newcommand{\qc}{\mathsf{qc}}

\newcommand{\op}{\mathsf{o}}
\newcommand{\ab}{\mathsf{Ab}}

\newcommand{\ZZ}{\mathbb{Z}}

\newcommand*{\longrightrightarrows}{
\ensuremath{%
\makebox[0pt][l]
{\raisebox{0.3ex}{\ensuremath{\longrightarrow}}}%
\raisebox{-0.3ex}{\ensuremath{\longrightarrow}}%
}}

\newcommand*{\longleftrightarrows}{
\ensuremath{%
\makebox[0pt][l]
{\raisebox{0.3ex}{\ensuremath{\longleftarrow}}}%
\raisebox{-0.3ex}{\ensuremath{\longrightarrow}}%
}}

\newcommand{\lto}{\longrightarrow}
\newcommand{\xto}{\xrightarrow}

\newcommand{\lot}{\longleftarrow}

\DeclareMathOperator{\iso}{\tilde{\to}}

\DeclareMathOperator{\liso}{\tilde{\lto}}
\DeclareMathOperator{\losi}{\tilde{\lot}}

\DeclareMathOperator{\Hom}{Hom}
\DeclareMathOperator{\shom}{\CH\mathit{om}}
\DeclareMathOperator{\rshom}{\R\!\shom}
\DeclareMathOperator{\dhom}{\pmb{\CH}\mathsf{om}}

\DeclareMathOperator{\spec}{Spec}

\DeclareMathOperator{\h}{H}
\DeclareMathOperator{\id}{id}
\DeclareMathOperator{\iid}{\mathsf{id}}
\DeclareMathOperator{\ev}{ev}

\DeclareMathOperator{\qco}{\mathsf{Qco}}

\DeclareMathOperator{\stack}{\mathsf{Stck}}

\DeclareMathOperator{\aff}{\mathsf{Aff}}

\DeclareMathOperator{\ff}{\mathsf{fppf}}

\newcommand{\GGamma}{\boldsymbol{\Gamma}}

\newcommand{\wadj}[1]{{#1}^\mathsf{a}}
\newcommand{\hv}{\mathsf{Hov}}
\newcommand{\chom}[1]{\mathsf{Hom}^{\cc}_{#1}}
\newcommand{\cdhom}[1]{\R\mathsf{Hom}^{\bullet\,\cc}_{#1}}

\newcommand{\ie}{{\it i.e.\/}}
\newcommand{\cfr}{{\it cf.\/}}
\newcommand{\lc}{{\it loc.~cit.\/} }

\begin{document}

\begin{abstract} 
Let $\SX$ be an Adams geometric stack. We show that $\D(\A_\qc(\SX))$, its derived category of quasi-coherent sheaves, satisfies the axioms of a stable homotopy category defined by Hovey, Palmieri and Strickland in \cite{hps}. Moreover we show how this structure relates to the derived category of comodules over a Hopf algebroid that determines $\SX$.
\end{abstract}

\maketitle
\tableofcontents

\section*{Introduction}
Let $X$ be a quasi-compact and semi-separated scheme. In \cite{asht} it is shown that its derived category of quasi-coherent sheaves $\D(\A_\qc(X))$ satisfies the axioms of a stable homotopy category from \cite{hps}. During the preparation of that paper we were asked whether the same result holds for the derived category of quasi-coherent sheaves on an algebraic stack. Unfortunately, the available references \cite{lmb} and \cite{O} suffered from some inaccuracies, and most importantly, did not contain the existence of generators in the category of quasi-coherent sheaves. To remedy this we embarked on a project of reconciling all the available definitions and settling the question of existence of generators. The project has born fruit in the form of \cite{gstck}.

In the latter paper we stick to the context of \emph{geometric stacks}, \ie{} those that are quasi-compact and possess an affine diagonal (in other words, they are semi-separated). This is a minimal requirement for our goal in view of the necessity of this hypothesis already in the scheme case. On the other hand, these stacks admit a representation by an \emph{affine} groupoid scheme, which corresponds by the algebra-geometry duality to an algebraic gadget called a Hopf algebroid. These objects play an important role in homotopy theory, in the context of orientable generalized cohomology theories and the main reference for this is \cite{R}. In this setting, quasi-coherent sheaves on a geometric stack correspond to comodules over the defining Hopf algebroid. This was used in a crucial way in \cite{goe} and \cite{hop} where the authors looked at the moduli stack of formal groups which is an \emph{ind}-geometric stack. In general, the main properties of quasi-coherent sheaves using comodules were studied by Hovey in \cite{hmh} and \cite{hov}.

In this paper, we will follow the general conventions and notations in \cite{lmb} and \cite{gstck}. Our setup differs from the one in \cite{stack}; this paper is essentially independent of it. Their treatment is more general due to the fact that their algebraic stacks have less strict  conditions on their diagonals. 

Let us explain some differences between the present approach and the one at \cite{stack}. In this paper $\D_\qc(\SX) := \D_\qc(\SX_{\ff}, \CO_{\SX})$ where $\SX_{\ff}$ denotes the so called small flat site. In Stacks Project, big flat sites relative to a choice of a small category of schemes are considered, see specifically \cite[Tags \href{http://stacks.math.columbia.edu/tag/06TH}{06TH}, \href{http://stacks.math.columbia.edu/tag/021R}{021R}]{stack}. This has the merit of simplifying the proof of functoriality of the corresponding categories of sheaves of modules and its derived counterparts. However, in the \cite{stack} setting, the category $\CO_{\SX}\md$ depends on this choice, because the size of the modules of sections over the objects of the site is bounded by the bigger cardinal available. By \cite[Tag \href{http://stacks.math.columbia.edu/tag/07B9}{07B9}]{stack}, $\D_\qc(\sch_{\ff}/\SX, \CO_{\SX})$ does not depend because it is equivalent to $\D_\qc(\SX_{{\mathsf{liss-\acute{e}t}}}, \CO_{\SX})$. This last category agrees with our $\D_\qc(\SX)$ because quasi-coherent sheaves are the same on both sites as they both correspond to Cartesian presheaves \cite[Theorem 3.12]{gstck}.

We have been asked about the choice of $\D(\A_\qc(\SX))$ instead of the usual $\D_\qc(\SX)$. A simple reason is that under our hypothesis there is an equivalence between $\D^+(\A_\qc(\SX))$ and  $\D^+_\qc(\SX)$  (Proposition \ref{Qequiv}). Notice that without the semi-separation hypothesis this equivalence need not hold.

It turns out that the existence of nice generators in the category of quasi-coherent sheaves is problematic. We have to impose the so called \emph{Adams condition}. This property of the category of comodules over a Hopf algebroid is equivalent to the classical resolution property on schemes (see the discussion in section 2). Under this additional hypothesis, one can prove the existence of dualizable generators and a structure of symmetric closed monoidal category on the derived category, thus fulfilling the axioms of Hovey, Palmieri and Strickland. Unlike the case of schemes, on geometric stacks the existence of \emph{compact} generators or, more precisely, the question whether dualizable complexes are compact is a delicate one. The failure is due to the existence of stacks with  infinite homological dimension, as the classifying stack of an algebraic group attests. In the case of finite homological dimension, Hall and Rydh proved that this difficulty does not arise \cite{hr}. One may think that cohomology of quasi-coherent sheaves on an algebraic stack generalizes both cohomology of quasi-coherent sheaves on a scheme and cohomology of representations on a group scheme.

In a sense this paper addresses the same problem as Hovey's \cite{hov}. The results are of a  different sort. In Hovey's words, he considers the stable \emph{homotopy} theory of comodules rather than its \emph{homology}. In practice, this means that he considers an a priori different localization of the categories of complexes of quasi-coherent sheaves on a stack. We ignore if both categories agree but see \ref{compHov} for a detailed discussion. In any case, we stress that our methods differ from those in \cite{hov}, as there homotopical algebra and model categories are used while in the present paper we employ just homological algebra and derived categories.

Let us discuss the contents of this paper. Let us recall first the axioms of stable homotopy category from \cite{hps}. Let $\ST$ be a category, we say that $\ST$ is a stable homotopy category if the following hold:
\begin{enumerate}
\item \label{tri} $\ST$ is a triangulated category.
\item \label{scl} $\ST$ is a symmetric closed category.
\item \label{sdg} $\ST$ possesses a system of \emph{strongly dualizable} generators.
\item \label{cop} $\ST$ possesses arbitrary coproducts.
\item \label{brw} A cohomological functor $\SF \colon \ST^{\op} \to \ab$ is representable, \ie{} there is a canonical isomorphism $\SF \cong \Hom_{\ST}(-,X)$ with $X \in \ST$.
\end{enumerate}

Fix from now on a geometric stack $\SX$. In the first section, using the fact that $\A_\qc(\SX)$ is a Grothendieck category, we settle the conditions (\ref{tri}),  (\ref{cop}) and (\ref{brw}). We also discuss the coherator functor $\SQ$ and its derived functor $\R\SQ$. In the case of schemes the latter induces an equivalence between the categories $\D_\qc(\SX)$ and  $\D(\A_\qc(\SX))$. However, on geometric stacks $\R\SQ$ might not be bounded, so the equivalence holds between the bounded-below derived categories (see Proposition \ref{Qequiv}). In the next section we discuss the Adams property and its characterization in terms of the existence of global resolutions by previous results of Hovey and Sch\"appi \cite{Sch}.

Section 3 deals with the symmetric closed structure. The Adams condition guarantees the existence of flat resolutions. This permits the extension of the monoidal structure from 
$\A_\qc(\SX)$ to $\D(\A_\qc(\SX))$. Then we define the internal Hom functor by applying the coherator functor to the internal Hom sheaf and deriving it by quasi-coherent injective resolutions. The $\otimes$-Hom adjunction follows. In the next section we address the question of the existence of dualizable generators. These are a suitable representative set of isomorphy classes of perfect complexes. We show that perfect complexes are strongly dualizable in Proposition \ref{perfsd}. Its proof differs greatly from the corresponding one in the case of schemes \cite[Proposition 4.4]{asht} because the coherator does not provide in this case an equivalence of categories. All in all, this completes the verification of the axioms of stable homotopy category.

In the last section we consider a geometric stack $\SX = \stack(A_\bullet)$, presented as the stack associated to a groupoid in affine schemes defined by a Hopf algebroid $A_\bullet$. We show that the equivalence of categories from \cite[Corollary 5.9]{gstck} preserves the symmetric closed structures between the categories $\A_\qc(\SX)$ and $A_\bullet\com$ of quasi-coherent sheaves on $\SX$ and $A_\bullet$-comodules, respectively. Along the way we give a purely algebraic description of the internal Hom in $A_\bullet\com$. The equivalence extends to the corresponding derived categories preserving their symmetric closed structures.

\begin{ack}
 We thank A. J. de Jong for illuminating conversations about the construction of big sites in \cite{stack}.
\end{ack}

\section{Basic results}

We will follow the notation and conventions in \cite{gstck}. In particular, we will denote by $\SX$ a geometric stack (\ie{} a semi-separated and quasi-compact Artin stack). Further, we will denote by $\SX_{\ff}$ the topos associated to its \emph{small flat site} $\aff_{\ff}/\SX$. Its objects are pairs $(V,v)$ with $V$ an affine scheme and $v \colon V \to \SX$ a flat finitely presented 1-morphism of stacks and its coverings are given by jointly surjective families of finitely presented flat maps. We recall that $\SX$ admits a presentation $p \colon U \to \SX$ where $U$ is an affine scheme and $p$ is a surjective smooth morphism, see \cite[\S 3.1]{gstck}. In this setting, $p$ is an affine morphism, see \lc{}\! Sometimes it will be enough to take a faithfully flat morphism of finite presentation. All our presentations will be of this kind.

\begin{cosa}
Notice that there is a canonically defined sheaf of rings on $\aff_{\ff}/\SX$, \ie{} a ring object in $\SX_{\ff}$, that we denote by $\CO_{\SX}$. There is a naturally associated category of sheaves of modules, $\CO_{\SX}\md$. As $\SX_{\ff}$ has enough points \cite[Remark after 3.6]{gstck}, $\CO_{\SX}\md$ is abelian, with exact directed limits and (a set of) generators, in other words, a Grothendieck category\footnote{For a proof, one may consult \cite[Theorem 18.1.6]{ks}, having in mind that $\CO_{\SX}$ is a $\ZZ_{\SX}$-algebra.}. Following our previous usage we will denote it by $\A(\SX) := \CO_{\SX}\md$ and its derived category by $\D(\SX)$.
\end{cosa}

\begin{cosa}
Inside of $\A(\SX)$ there is a full abelian subcategory of quasi-coherent $\CO_{\SX}$-Modules on $\SX_{\ff}$ denoted $\qco(\SX)$ in \cite{gstck}. Notice that it has two possible descriptions, as per the definition through local presentations \cite[1.6]{gstck} or the equivalent characterization as \emph{Cartesian presheaves}, \cite[1.3]{gstck} \ie{} presheaves whose restriction map induces an isomorphism after base change. The agreement of both notions is proved in \cite[Theorem 3.12]{gstck}. From now on we will denote this category as $\A_\qc(\SX) := \qco(\SX)$  and its derived category by $\D(\A_\qc(\SX))$.
\end{cosa}

We quote the following result. It will unlock many of the features of the derived category of the category of quasi-coherent $\CO_{\SX}$-Modules on the small flat site of a geometric stack.

\begin{thm}
For a geometric stack $\SX$, $\A_\qc(\SX)$ is a Grothendieck category.
\end{thm}

\begin{proof}
See \cite[Corollary 5.10]{gstck}.
\end{proof}

\begin{thm}\label{asht145}
 The category $\D(\A_\qc(\SX))$ satisfies the axioms (\ref{tri}),  (\ref{cop}) and (\ref{brw}) from \cite[1.1]{asht} or \cite[1.1]{hps}.
\end{thm}

\begin{proof}
These assertions are consequences of the fact that $\A_\qc(\SX)$ is a Grothen\-dieck category. 

The statement (\ref{tri}) is trivial because a derived category is triangulated by construction (\cfr{} \cite[Example (1.4.4)]{yellow}). The existence of coproducts is due to existence and exactness of coproducts in an AB5 category. And (\ref{brw}) is satisfied because a cohomological functor from the derived category of a Grothendieck category to $\ab$ is representable as follows from \cite[Theorem 5.8]{AJS}.
\end{proof}


\begin{cosa} \textbf{The coherator functor on geometric stacks}.
 Let us consider the inclusion functor $\iota \colon \A_\qc(\SX) \to \A(\SX)$. We will construct a right adjoint, the coherator, $\SQ_{\SX} \colon \A(\SX) \to \A_\qc(\SX)$. 
 For an affine scheme the construction of the coherator is simple. Let $V$ be an affine scheme and $\CF \in \A(V)$. We take $\SQ_V \CF := \widetilde{\Gamma(V, \CF)}$ and $\alpha_{\CF} \colon\SQ_V \CF \to \CF$ the canonical map. It is well-known that in this case we have an adjunction $\iota \dashv \SQ_V$ \cite[Lemme 3.2]{erg}. For any affine scheme $V$  we will abbreviate $\SQ := \SQ_{V}$.

 Now, let $p \colon U \to \SX$ be a presentation of the geometric stack $\SX$. Recall that both $p^*$ and $p_*$ preserve quasi-coherence (\cite[Proposition 6.16 and 6.17]{gstck}, respectively).  Consider the Cartesian diagram:
 \begin{equation*}
\begin{tikzpicture}[baseline=(current  bounding  box.center)]
\matrix(m)[matrix of math nodes, row sep=2.6em, column sep=2.8em,
text height=1.5ex, text depth=0.25ex]{
  U\times_{\SX}U & U \\
  U              & \SX \\};
\path[->,font=\scriptsize,>=angle 90] (m-1-1) edge
node[auto] {$p_2$}
(m-1-2) edge node[left] {$p_1$} (m-2-1)
(m-1-2) edge node[auto] {$p$} (m-2-2)
(m-2-1) edge node[auto] {$p$} (m-2-2);
\draw [shorten >=0.2cm,shorten <=0.2cm, ->, double] (m-2-1) -- (m-1-2) node[auto, midway,font=\scriptsize]{$\phi$};
\end{tikzpicture}
\end{equation*}

By virtue of \cite[Corollary 6.15]{gstck}, there are natural isomorphisms 
\[
\phi^* \colon p_2^*p^* \to p_1^*p^* \text{ and } 
\phi_* \colon p_*p_{1*} \to p_*p_{2*}~.
\]
For $\CF \in \A(\SX)$, let $\SQ_{\SX} \CF$ be defined as the equalizer of the lower row of the following commutative diagram:

\[
\begin{tikzpicture}
\matrix (m) [matrix of math nodes, row sep=3em, column sep=4em]{
\CF     &   p_*p^* \CF\,      &  p_*p_{2~*}p_1^*p^* \CF \\
\SQ_{\SX} \CF &   p_*\SQ p^*\CF     &  p_*p_{2~*}\SQ p_1^*p^* \CF . \\
}; 
\draw[<-,font=\scriptsize, dashed] 
(m-1-1) -- node[left]{$\alpha_{\CF}$}(m-2-1) ;
\draw [transform canvas={yshift= 0.3ex},font=\scriptsize,->]
(m-1-2) -- node[above]{} (m-1-3);
\draw [transform canvas={yshift=-0.3ex},font=\scriptsize,->]
(m-1-2) -- node[below]{} (m-1-3);
\draw [transform canvas={yshift= 0.3ex},font=\scriptsize,->]
(m-2-2) -- node[above]{} (m-2-3);
\draw [transform canvas={yshift=-0.3ex},font=\scriptsize,->]
(m-2-2) -- node[below]{} (m-2-3);
\draw[<-,font=\scriptsize] (m-1-2) -- node[left]{$p_*\alpha_{p^*\CF}$}(m-2-2); 
\draw[<-,font=\scriptsize] (m-1-3) -- node[right]{$p_*p_{2~*}\alpha_{p_1^*p^*\CF}$}(m-2-3);
\draw[->,font=\scriptsize]
(m-1-1) edge node[auto] {} (m-1-2);
\draw[->,font=\scriptsize]
(m-2-1) edge node[auto] {} (m-2-2);
\end{tikzpicture}
\]
The top horizontal maps on the right are the compositions of the morphisms $p_*p^* \to p_*p_{i~*} p_i^*p^*$ defined via the units of the adjunctions $p_i^* \dashv p_{i\,*}$, with the natural isomorphisms 
\[
p_*p_{i~*} p_i^*p^* \liso p_*p_{2~*}p_1^*p^*
\]
induced by $\phi^*$ and $\phi_*$, where $i \in \{1, 2\}$.

The bottom horizontal maps are given by the same units composed now with the following composites of natural maps 
\[
 p_*p_{i~*} p_i^*\SQ p^* \lto p_*p_{i~*} \SQ p_i^* p^* \liso p_*p_{2~*} \SQ p_1^*p^*
\]
again $i \in \{1, 2\}$. Notice that the first map need not be an isomorphism. The map $\alpha_{\CF}$ is the one induced between both equalizers and it is an isomorphism whenever $\CF$ is quasi-coherent. It is straightforward to check that this construction yields an adjunction $\iota \dashv \SQ_{\SX}$ and $\alpha$ is the counit.
\end{cosa}

\begin{rem}
 Let $\SST = p_*p^*$. Recall the equalizer diagram \cite[Lemma 3.11]{gstck}:
 \[
\CF \longrightarrow \SST \CF \, \longrightrightarrows \, \SST^2 \CF
\]
Notice that it corresponds to the top row of the previous diagram by the base change isomorphism $p_{2~*}p_1^* \cong p^*p_*$. 
\end{rem}

The functor $\iota \colon \A_\qc(\SX) \to \A(\SX)$ is exact, while $\SQ_{\SX}$ is just left-exact. They induce an adjunction
\[
\D(\A_\qc(\SX)) \,
\underset{\iota}{\overset{\R\SQ_{\SX}}{\longleftrightarrows}} \,
\D(\SX)
\]
between the derived categories. Denote, as usual, by $\D_\qc(\SX)$ the full subcategory of $\D(\SX)$ formed by those complexes whose homology is quasi-coherent. We have the following

\begin{prop}\label{Qequiv}
 The previous adjunction induces an equivalence of categories $\D^+_\qc(\SX) \cong \D^+(\A_\qc(\SX))$.
\end{prop}

\begin{proof}
 It is enough to check that both the unit and counit of the adjunction restricted to the bounded below categories are equivalences. Now we can use the way-out Lemma \cite[Lemma (1.11.3)]{yellow} because both $\iota$ and $\R\SQ_{\SX}$ are bounded below, the former because it comes from an exact functor between abelian categories and the latter is a right derived funtor of a left exact funtor between abelian categories. Thus, by way-out we are reduced to check that $\R\SQ_{\SX} \CF \cong \CF$ for $\CF \in \A_\qc(\SX)$. Choose a  presentation $p \colon U \to \SX$ with $U$ an affine scheme. 

Let $\CG \in \A_\qc(U)$ and $\CG \to \CI$ be an injective resolution of $\CO_{U}$-Modules. By the exactness of $p^*$, we have that $p_*\CI$ is a complex of injective $\CO_{\SX}$-Modules. For all $(V,v) \in \aff_{\ff}/\SX$, it holds that $H^i(V \times_{\SX} U, \CG) = 0, \forall i > 0$, since $V \times_{\SX} U$ is an affine scheme \cite[\S 3.1]{gstck}. We conclude that $p_*\CI$ is a resolution of $p_*\CG$. As a consequence,
\[
\SR^i\SQ_{\SX} p_*\CG \cong
H^i(\SQ_{\SX} p_*\CI) \cong
H^i(p_*\SQ \CI) \cong
p_*H^i(\widetilde{\Gamma(U, \CI)}) \cong
p_*\widetilde{\h^i(U, \CG)} \cong 0
\]
from which we conclude that the complexes in the image of $p_*$ are acyclic for $\SQ_{\SX}$.
For $\CF \in \A_\qc(\SX)$, let $\SST^\bullet \CF$ be the complex
\[
\cdots \lto 0 \lto \SST \CF \lto \SST^2 \CF \lto \SST^3 \CF \lto \cdots .
\]
Notice that this yields a resolution $\CF \to \SST^\bullet \CF$ by $\SQ_{\SX}$-acyclic sheaves, therefore 
$\R\SQ_{\SX} \CF \cong \SQ_{\SX}\SST^\bullet \CF$ but the complex $\SST^\bullet \CF$ is formed by quasi-coherent sheaves, so $\SQ_{\SX}\SST^\bullet \CF = \SST^\bullet \CF \cong \CF$ and we are done.
\end{proof}

\begin{rem}
This statement is also proved in \cite[Theorem 3.8]{lur}, but the present proof is different. It is inspired by \cite[Proposition B.16]{tt}.
\end{rem}

\section{The resolution property and Adams stacks}

Let $\SX = \stack(A_\bullet)$ be a geometric stack, \ie{} a quasi-compact semi-separated algebraic stack. By choosing a presentation $p \colon U \to \SX$ with $U := \spec(A_0)$ an affine scheme and $A_1$ such that $U \times_{\SX} U = \spec(A_1)$, the couple $A_\bullet := (A_0, A_1)$ is a flat Hopf algebroid. See, for instance, \cite[\S 5.1]{gstck}. For a general discussion of Hopf algebroids, see \cite[Appendix A1]{R}.

\begin{cosa}
According to \cite[Definition 1.4.3]{hov}, $\SX$ satisfies the \emph{Adams condition} whenever $A_1$ is the filtered colimit of comodules whose underlying $A_0$-module is projective and finitely generated. 

We may rephrase the definition as follows: an Adams geometric stack is one such that, for any presentation $p \colon U \to \SX$ by an affine scheme, the sheaf $p_* \CO_U$ is a filtered direct limit of finitely generated locally free sheaves on $\SX$. Indeed, through the equivalence of categories between $A_\bullet$-comodules and $\A_\qc(\SX)$ \cite[Corollary 5.9 and Proposition 5.12]{gstck}, the comodule corresponding to the sheaf 
$p_* \CO_U$ is $(A_1, \nabla)$. Also, a comodule whose underlying $A_0$-module is projective and finitely generated corresponds to a finitely generated locally free sheaf.

 The relevance of the Adams condition for the study of quasi-coherent sheaves on a geometric stack was put forth by Hovey in \cite{hmc} motivated by considerations from homotopy theory.
\end{cosa}
 
\begin{cosa}
Let us recall the version of the \emph{resolution property}\footnote{Sometimes called also the ``strong resolution property''.} appropriate for a non-Noetherian situation. An algebraic stack (or scheme) possesses the resolution property if and only if every quasi-coherent sheaf is a quotient of a coproduct of finitely generated locally free sheaves, \ie{} the category of quasi-coherent sheaves is generated by the collection of finitely generated locally free sheaves. 

The Adams condition may seem at first somewhat technical, but by a theorem of Sch\"appi it is equivalent to the resolution property, let us state it precisely:
\end{cosa}

\begin{thm}\label{aisrp}
A geometric stack is Adams if and only if the resolution property holds.
\end{thm}

\begin{proof}
 This is a restatement of \cite[Theorem 1.3.1]{Sch} where the ``only if'' part is proved. The ``if'' part is already contained in \cite[Proposition 1.4.4]{hov}.
\end{proof}

In the case of quasi-compact semi-separated schemes, the Adams condition may be expressed as follows. Let $X$ be such a scheme and $X = \cup_{i = 1}^n U_i$ an affine open cover. Let $U := \coprod_{i = 1}^n U_i$ and $p \colon U \to X$ the canonical morphism. Then every quasi-coherent sheaf is a quotient of a coproduct of finitely generated locally free sheaves if and only if $p_*\CO_U$ is.

\begin{cosa}
The resolution property or Adams condition is not easy to characterize even in the case of schemes. It is clear that quasi-projective or, more generally, quasi-compact divisorial schemes possess this property. For a thorough treatment of what was known in 2004, see \cite{to}. More recently, it is worth mention that Gross has proved that any separated algebraic surface has the resolution property \cite{gr}. On algebraic stacks, Edidin, Hassett, Kresch and Vistoli \cite{4mag} proved that the resolution property forces a stack
to be a quotient stack, \ie{}\! it admits a presentation of the form $[X/G]$ where $X$ is an algebraic space of finite type over some Noetherian base scheme $S$ with an action of a affine group scheme $G$. In \cite{kr}, Kresch proves that what he calls \emph{quasi-projective stacks} satisfy the resolution property. These are finite type Deligne-Mumford stacks over a characteristic zero field $K$ that admit a (locally) closed embedding into a smooth Deligne-Mumford stack which is proper over $\spec(K)$ and has projective coarse moduli space.
\end{cosa}

\begin{rem}
 As it is well known, on algebraic stacks (as well as on schemes) is almost never the case that the category $\A_\qc(\SX)$ has enough projectives\footnote{In fact this is equivalent to the stack (or scheme) being an affine scheme.}. Thus, to show the existence of left derivatives of important functors, like the tensor product we will have to resort to appropriate acyclic resolutions. We will explain this in the next sections.
\end{rem}

\section{Closed structure}

\begin{cosa}
Let $\SX$ be a geometric stack. We begin by recalling that $\A(\SX)$ is a closed category. First for any two sheaves of modules $\CF, \CG \in \A(\SX)$ one has a sheaf $\CF \otimes_{\CO_{\SX}} \CG$ defined as the sheaf on $\aff_{\ff}/\SX$ associated to the presheaf:
\[
(V,v) \leadsto \CF(V,v) \otimes_{\CO_V(V,v)} \CG(V,v)
\]
see \cite[IV, Proposition 12.10]{sga41}. It is an $\CO_{\SX}$-Module by commutativity, as usual.

Similarly, the presheaf
\[
(V,v) \leadsto  \Hom_{\CO_V}(\CF|_V, \CG|_V)
\]
is a sheaf by \cite[IV, Proposition 12.1]{sga41}. We will denote it by $\shom_{\SX}(\CF,\CG)$.

All the usual properties that give the category a symmetric closed structure are satisfied as in the case of schemes \cite[(3.4.1)]{yellow}, specially the adjunction for $\CF, \CG, \CH \in \A(\SX)$, see \cite[12.8]{sga41},
\begin{equation}\label{adjsga}
\Hom_{\A(\SX)}(\CF \otimes_{\CO_{\SX}} \CG, \CH) \liso
\Hom_{\A(\SX)}(\CF, \shom_{\SX}(\CG,\CH))
\end{equation}
that holds internally, too, \ie{}\! with the external $\Hom$ replaced by the homomorphism sheaf. The unit object for the monoidal structure is $\CO_{\SX}$ as expected.
\end{cosa}

\begin{prop}\label{qcten}
Let $\SX$ be a geometric stack, and let $\CF, \CG \in \A_\qc(\SX)$. Then, $\CF \otimes_{\CO_{\SX}} \CG \in \A_\qc(\SX)$.
\end{prop}

\begin{proof}
As Cartesian presheaves are already sheaves \cite[Lemma 3.7]{gstck} and quasi-coherent sheaves agree with Cartesian presheaves \cite[Theorem 3.12]{gstck}, it is enough to check the condition of being Cartesian on the corresponding presheaves.

Let $(f, \alpha) \colon (W,w) \to (V,v)$ a morphism in $\aff_{\ff}/\SX$, and let $W:= \spec(B)$ and $V := \spec(C)$. Associated to the restriction
\[
(\CF \otimes_{\CO_{\SX}} \CG)(f, \alpha) \colon 
    \CF(V,v) \otimes_{C} \CG(V,v) \lto 
    \CF(W,w) \otimes_{B} \CG(W,w)
\]
we have a morphism (notation as in \cite[1.3]{gstck})
\[
\wadj{((\CF \otimes_{\CO_{\SX}} \CG)(f, \alpha))} \colon 
    B \otimes_{C} (\CF(V,v) \otimes_{C} \CG(V,v)) \lto 
    \CF(W,w) \otimes_{B} \CG(W,w)
\]
To see that $\CF \otimes_{\CO_{\SX}} \CG$ is cartesian we have to check that this last map is an isomorphism. The result follows from the commutativity of the following diagram.
\[
\begin{tikzpicture}
      \draw[white] (0cm,2cm) -- +(0: \linewidth)
      node (G) [black, pos = 0.25] {$B \otimes_{C} (\CF(V,v) \otimes_{C} \CG(V,v))$}
      node (H) [black, pos = 0.75] {$\CF(W,w) \otimes_{B} \CG(W,w)$};
      \draw[white] (0cm,0.25cm) -- +(0: \linewidth)
      node (E) [black, pos = 0.5] {$(B \otimes_{C} (\CF(V,v)) \otimes_{B} (B\otimes_{C} \CG(V,v))$};
      \draw [->] (G) -- (H) node[above, midway, scale=0.75]{$\wadj{((\CF \otimes_{\CO_{\SX}} \CG)(f, \alpha))}$};
      \draw [->] (G) -- (E) node[auto, swap, midway, scale=0.75]{$w$};
      \draw [->] (E) -- (H) node[auto, swap, midway, scale=0.75]{$\wadj{(\CF(f, \alpha))}\otimes_{B}\wadj{(\CG(f, \alpha))}$};  
\end{tikzpicture}
\]
where $w$ is the canonical isomorphism given by $w(b \otimes m \otimes n) := b \otimes m \otimes 1 \otimes n$.
\end{proof}

\begin{cosa}
 As on any closed category, we say that a quasi-coherent sheaf $\CF$ on a geometric stack $\SX$ is \emph{flat}, if the functor $\CF \otimes_{\CO_{\SX}} -$ is exact.
\end{cosa}

\begin{lem}\label{locisflat}
Let $\SX$ be a geometric stack. If $\CE$ is a locally free finite-type $\CO_{\SX}$-Module then $\CE$ is flat.
\end{lem}

\begin{proof}
Let  $p \colon U \to \SX$ be a presentation. The assertion that $\CE$ is a locally free finite-type $\CO_{\SX}$-Module amounts to say that $p^*\CE$ is a locally free finite-type $\CO_{U}$-Module. It follows that $\CE(U,p)$ is a flat $A_0$-module.

We want to check that $\CE(V,v)$ is a flat $\CO_{\SX}(V,v)$-module for every $(V,v) \in \aff_{\ff}/\SX$. Put $V = \spec(B)$. If $(V,v)$ factors through $(U,p)$, the result is clear, by the Cartesian property $\CE(V,v) \cong B \otimes_{A_0} \CE(U, p)$ since flatness is preserved by base change. Otherwise, we have a pull-back square 
\begin{equation*}
\begin{tikzpicture}[baseline=(current  bounding  box.center)]
\matrix(m)[matrix of math nodes, row sep=2.6em, column sep=2.8em,
text height=1.5ex, text depth=0.25ex]{
  V\times_{\SX}U & U \\
  V              & \SX \\};
\path[->,font=\scriptsize,>=angle 90] 
(m-1-1) edge node[auto] {$v'$}
(m-1-2) edge node[left] {$p'$} (m-2-1)
(m-1-2) edge node[auto] {$p$} (m-2-2)
(m-2-1) edge node[auto] {$v$} (m-2-2);
\draw [shorten >=0.2cm,shorten <=0.2cm, ->, double] (m-2-1) -- (m-1-2) node[auto, midway,font=\scriptsize]{$\phi$};
\end{tikzpicture}
\end{equation*}
Let $V\times_{\SX}U = \spec(B')$ \cite[\S 3.1]{gstck}. By the previous argument $\CE(V\times_{\SX}U,vp')$ is a flat $B'$-module. Now, by the Cartesian property of $\CE$, we have that $\CE(V\times_{\SX}U,vp') \cong B' \otimes_{B} \CE(V,v)$ and we conclude because $p'$ is faithfully flat.
\end{proof}

\begin{rem}
 Lurie presents another approach of this result in \cite[Example 5.8.]{lur} and discusses the relation between local flatness and exactness of tensor product for an algebraic stack.
\end{rem}

\begin{cosa}
Denote by $\K(\A_\qc(\SX))$ the homotopy category of complexes of objects in $\A_\qc(\SX)$.
 Following \cite[\S 2.5]{yellow} we say that a complex $\CP \in \K(\A_\qc(\SX))$ is $q$-flat if given an acyclic complex $\CA \in \K(\A_\qc(\SX))$, $\CP \otimes_{\CO_{\SX}} \CA$ is also acyclic.
\end{cosa}

\begin{prop} \label{qflatres}
Every complex in $\K(\A_\qc(\SX))$ on an Adams geometric stack $\SX$ is quasi-isomorphic to a $q$-flat complex, in other words, has a $q$-flat resolution. 
\end{prop}
 
\begin{proof}
By Theorem \ref{aisrp} and Lemma \ref{locisflat} every object can be covered by a flat quasi-coherent sheaf. Now, by the usual step by step procedure we see that every complex in $\K^-(\A_\qc(\SX))$ has a (bounded above) resolution by flat sheaves, therefore a $q$-flat resolution, see \cite[Proposition (2.5.5)]{yellow}. By taking limits (or just homotopy limits imitating the procedure in \cite[Proposition 4.3]{AJS}) we extend the result to unbounded complexes of quasi-coherent sheaves, having in mind that a direct limit (or a coproduct) of flat sheaves remains flat.
\end{proof}

Now, let us transport the monoidal structure of $\A_\qc(\SX)$ to the derived category of quasi-coherent sheaves. 

\begin{cor}
 For an Adams geometric stack $\SX$ there is a bifunctor
\[
-\otimes^{\LL}_{\CO_{\SX}}\!\!\!- \colon 
\D(\A_\qc(\SX)) \times \D(\A_\qc(\SX)) \to \D(\A_\qc(\SX))
\]
called the derived tensor product.
\end{cor}

\begin{proof}
Use the existence of $q$-flat resolutions (Proposition \ref{qflatres}). The fact that acyclic resolutions yield derived functors is explained in \cite[(2.5.7)]{yellow}.
\end{proof}

\begin{rem}
Let $\SX$ be an Adams geometric stack.
\noindent
\begin{enumerate}
 \item The usual balancing property of Tor follows from the fact that for $\CF, \CG \in \K(\A_\qc(\SX))$ and $\CP_\CF \to \CF$, $\CP_\CG \to \CG$ the corresponding $q$-flat resolutions we have the following quasi-isomorphisms
\[
\CP_\CF \otimes_{\CO_{\SX}} \CG \liso
\CF \otimes^{\LL}_{\CO_{\SX}} \CG \losi
\CF \otimes_{\CO_{\SX}} \CP_\CG.
\]
 \item  As the derived tensor product is obtained much in the same way as in well-known contexts (like sheaves of modules on a scheme, quasi-coherent sheaves on a quasi-compact quasi-separated scheme, etc.) the usual properties that make the triple 
 $(\D(\A_\qc(\SX)), \otimes^{\LL}_{\CO_{\SX}}, \CO_{\SX})$
 a monoidal category hold. To complete the \emph{closed} structure we need a little extra work.
\end{enumerate}
\end{rem}


\begin{cosa} \textbf{The internal \emph{hom}.}
Let $\SX$ be a geometric stack. In general for $\CF, \CG \in \A_\qc(\SX)$ it is not guaranteed that $\shom_{\SX}(\CF,\CG) \in \A_\qc(\SX)$ unless we impose strong finiteness conditions on $\CF$. To define a closed monoidal structure, we have to apply the coherator functor to the internal \emph{hom} in $\A(\SX)$ and we get the bifunctor $\SQ_{\SX}\!\shom_{\SX}(\CF, \CG)$ which takes values in $\A_\qc(\SX)$. 
\end{cosa}

\begin{prop}\label{adjtenhom}
There is an adjunction isomorphism
\[
 \Hom_{\A_\qc(\SX)}(\CF, \SQ_{\SX}\!\shom_{\SX}(\CG, \CH)) \liso
 \Hom_{\A_\qc(\SX)}(\CF \otimes_{\CO_{\SX}} \CG, \CH).
\]
\end{prop}

\begin{proof}
 It is obtained as the composition of the following chain of isomorphisms
 \begin{align*}
\Hom_{\A_\qc(\SX)}(\CF \otimes_{\CO_{\SX}} \CG, \CH) & = \Hom_{\A(\SX)}(\CF \otimes_{\CO_{\SX}} \CG, \CH)
                          \tag{$\A_\qc(\SX)$ is full in $\A(\SX)$}\\                            
    & \cong \Hom_{\A(\SX)}(\CF, \shom_{\SX}(\CG, \CH)) 
                          \tag{\ref{adjsga}} \\
    & \cong \Hom_{\A_\qc(\SX)}(\CF, \SQ_{\SX}\!\shom_{\SX}(\CG, \CH))
                          \tag{$\iota \dashv \SQ_{\SX}$}
 \end{align*}
The readers will check easily that these isomorphisms are natural.
\end{proof}

\begin{lem}\label{lema4}
 Let $\SX$ be an Adams geometric stack, and let $\CI \in \K(\A_\qc(\SX))$ be a $q$-injective complex. Then, the functor $\SQ_{\SX}\!\shom^\bullet_{\SX}(-, \CI)$ preserves quasi-isomorphisms in $\K(\A_\qc(\SX))$.
\end{lem}

\begin{proof}
 It is enough to prove that if the complex $\CF \in \K(\A_\qc(\SX))$ is acyclic, then so is $\SQ_{\SX}\!\shom^\bullet_{\SX}(\CF, \CI)$. Let $\CE$ be an $\CO_{\SX}$-Module finitely generated and locally free. Take $n \in \ZZ$. Notice that
\[
\Hom_{\K(\A_\qc(\SX))}(\CE[n], \SQ_{\SX}\!\shom^\bullet_{\SX}(\CF, \CI)) \cong
\Hom_{\K(\A_\qc(\SX))}(\CE[n] \otimes_{\CO_{\SX}} \CF, \CI)  = 0
\]
because $\CE[n] \otimes_{\CO_{\SX}} \CF$ is acyclic. But the collection of finitely generated and locally free sheaves generate $\A_\qc(\SX)$ (Theorem \ref{aisrp}), therefore
$\SQ_{\SX}\!\shom^\bullet_{\SX}(\CF, \CI)$ is acyclic.
\end{proof}

\begin{cosa}
For $\CF, \CG \in \D(\A_\qc(\SX))$, we define:
\[
\dhom^\bullet_{\SX}(\CF, \CG) := \R{}\SQ_{\SX}\!\shom^\bullet_{\SX}(\CF, \CG)
\]
\ie{} the derived functor (on the second variable) \emph{in} $\D(\A_\qc(\SX))$ of the composite functor $\SQ_{\SX}\!\shom^\bullet_{\SX}$. Notice that it is a $\Delta$-bifunctor in two variables.
\end{cosa}

\begin{lem}\label{lema6}
 Let $\SX$ be a geometric stack. If $\CP \in \K(\A_\qc(\SX))$ is $q$-flat and $\CI \in \K(\A_\qc(\SX))$ is $q$-injective, then $\SQ_{\SX}\!\shom^\bullet_{\SX}(\CP, \CI)$ is $q$-injective in $\K(\A_\qc(\SX))$.
\end{lem}

\begin{proof}
 Let $\CF \in \K(\A_\qc(\SX))$ be an acyclic complex. Note that by Proposition \ref{adjtenhom}
 \[
\Hom_{\K(\A_\qc(\SX))}(\CF, \SQ_{\SX}\!\shom^\bullet_{\SX}(\CP, \CI)) \cong
\Hom_{\K(\A_\qc(\SX))}(\CF \otimes_{\CO_{\SX}} \CP, \CI) = 0,
 \]
 since the complex $\CF \otimes_{\CO_{\SX}} \CP$ is exact. We conclude by \cite[Proposition (2.3.8)]{yellow} that $\SQ_{\SX}\!\shom_{\SX}(\CP, \CI)$ is $q$-injective.
\end{proof}

\begin{prop}\label{adjexttor}
Let $\SX$ be an Adams geometric stack and let $\CF, \CG, \CH \in \D(\A_\qc(\SX))$. There is a natural isomorphism in $\D(\A_\qc(\SX))$
\[
 \Hom_{\D(\A_\qc(\SX))}(\CF, \dhom^\bullet_{\SX}(\CG, \CH)) \liso
 \Hom_{\D(\A_\qc(\SX))}(\CF \otimes^{\LL}_{\CO_{\SX}} \CG, \CH).
\]
This establishes an adjunction $-\otimes^{\LL}_{\CO_{\SX}}\CG \dashv \dhom^\bullet_{\SX}(\CG, -)$ in $\D(\A_\qc(\SX))$.
\end{prop}

\begin{proof}
Choose $\CP_{\CG} \to \CG $ a $q$-flat resolution and $\CH \to \CI_{\CH}$ a $q$-injective resolution. Let $\D$ stand for $\D(\A_\qc(\SX))$ and $\K$ for $\K(\A_\qc(\SX))$. Consider the following chain of isomorphisms
\begin{align*}
\Hom_{\D}(\CF, \dhom^\bullet_{\SX}(\CG, \CH))
    & = \Hom_{\D}(\CF, \R{}\SQ_{\SX}\!\shom^\bullet_{\SX}(\CG, \CH)) 
                           \\                            
    & \cong \Hom_{\D}(\CF, \SQ_{\SX}\!\shom^\bullet_{\SX}(\CG, \CI_{\CH}))
                           \\
    & \cong \Hom_{\D}(\CF, \SQ_{\SX}\!\shom^\bullet_{\SX}(\CP_{\CG}, \CI_{\CH}))
                          \tag{Lemma \ref{lema4}}\\
    & \cong \Hom_{\K}(\CF, \SQ_{\SX}\!\shom^\bullet_{\SX}(\CP_{\CG}, \CI_{\CH}))
                          \tag{Lemma \ref{lema6}}\\
    & \cong \Hom_{\K}(\CF \otimes_{\CO_{\SX}} \CP_{\CG}, \CI_{\CH})
                          \tag{Proposition \ref{adjtenhom}}\\
    & \cong \Hom_{\D}(\CF \otimes_{\CO_{\SX}} \CP_{\CG}, \CH)
    \\
    & \cong \Hom_{\D}(\CF \otimes^{\LL}_{\CO_{\SX}} \CG, \CH)
\end{align*}
The composed isomorphism is the one we were looking for.
\end{proof}

\begin{cor}
Let $\SX$ be an Adams geometric stack and $\CF, \CG, \CH \in \D(\A_\qc(\SX))$, then the \emph{internal} adjunction holds, \ie{} there is a natural isomorphism in $\D(\A_\qc(\SX))$
\[
 \dhom^\bullet_{\SX}(\CF, \dhom^\bullet_{\SX}(\CG, \CH)) \iso
 \dhom^\bullet_{\SX}(\CF \otimes^{\LL}_{\CO_{\SX}} \CG, \CH).
\]
\end{cor}

\begin{proof}
It is a consequence of the axioms of a symmetric monoidal closed category, see \cite[Exercise (3.5.3) (e)]{yellow}.
\end{proof}

\begin{thm}\label{asht2}
 Let $\SX$ be an Adams geometric stack. The category $\D(\A_\qc(\SX))$ has a natural structure of symmetric monoidal closed category, in other words, axiom (\ref{scl}) of \cite[1.1]{hps} holds.
\end{thm}

\begin{proof}
As we remarked $(\D(\A_\qc(\SX)), \otimes^{\LL}_{\CO_{\SX}}, \CO_{\SX})$ constitutes a monoidal category. 

The internal hom is well behaved, namely, the adjunction 
\[-\otimes^{\LL}_{\CO_{\SX}}\CG \dashv \dhom^\bullet_{\SX}(\CG, -)\]
 for every $\CG \in \D(\A_\qc(\SX))$, holds by Proposition \ref{adjexttor}. It is clear that both bifunctors are $\Delta$-functors in either variable. Finally, the diagram
\begin{center}
\begin{tikzpicture}
\matrix(m)[matrix of math nodes, row sep=2.6em, column sep=2.8em,
text height=1.5ex, text depth=0.25ex]{
\CO_{\SX}[r] \otimes^{\LL}_{\CO_{\SX}} \CO_{\SX}[s] & \CO_{\SX}[r + s] \\
\CO_{\SX}[s] \otimes^{\LL}_{\CO_{\SX}} \CO_{\SX}[r] & \CO_{\SX}[r + s] \\
 };
\path[->,font=\scriptsize,>=angle 90]
(m-1-1) edge node[auto] {$\sim$} (m-1-2)
        edge node[left] {$\theta$} (m-2-1)
(m-1-2) edge node[auto] {$(-1)^{rs}$} (m-2-2)
(m-2-1) edge node[auto] {$\sim$} (m-2-2);
\end{tikzpicture}
\end{center}
with $\theta$ defined as in \cite[(1.5.4.1)]{yellow} is commutative. Note that $\CO_{\SX}$ is $q$-flat considered as a complex concentrated in degree 0. The map $\theta$ corresponds to $T$ in \cite[Definition A.2.1(4)]{hps}.
\end{proof}

\section{Dualizable generators}

\begin{cosa}
Let $\SX$ be a geometric stack. A complex $\CF \in \D(\A_\qc(\SX))$ is called \emph{strongly dualizable} if the canonical map
\[
\dhom^\bullet_{\SX}(\CF, \CO_{\SX}) \otimes^{\LL}_{\CO_{\SX}} \CG \lto
\dhom^\bullet_{\SX}(\CF, \CG)
\]
is an isomorphism for all $\CG \in \D(\A_\qc(\SX))$ \cite[Definition 1.1.2]{hps}.

A complex $\CE \in \K(\A_\qc(\SX))$ is called \emph{strictly perfect} if it is a bounded complex of locally free finitely generated modules. We say that $\CE$ is \emph{perfect} if it is locally (for the small flat topology) quasi-isomorphic to a strictly perfect complex. Observe that $\CE$ is perfect if, and only if, for any presentation $p \colon U \to \SX$, the complex $p^*\CE$ is quasi-isomorphic to a strictly perfect complex.

Notice that on an Adams geometric stack every perfect complex is isomorphic to a strictly perfect complex in $\D(\A_\qc(\SX))$. The interested reader may adapt the proof in \cite[Proposition 2.2.9 (b)]{erg} or \cite[Proposition 2.3.1 (d)]{tt}.
\end{cosa}

\begin{prop}
In this setting, if $\CE$ is strictly perfect, then $\CE$ is $q$-flat.
\end{prop}

\begin{proof}
By Lemma \ref{locisflat}, $\CE$ is a bounded complex of flat sheaves, therefore it is $q$-flat.
\end{proof}

We have the following analogue of \cite[Theorem 2.4.1 (c)]{tt}.

\begin{lem}\label{perfqchom}
If $\CE$ is a strictly perfect complex and $\CH$ quasi-coherent, it follows that $\shom^\bullet_{\SX}(\CE, \CH) \in \K(\A_\qc(\SX))$.
\end{lem}

\begin{proof}
It is a ($\ff$-)local question, so the proof follows similar lines that the aforementioned result. Let $p \colon U \to \SX$ be a presentation, it is enough to check that $p^*\!\shom^\bullet_X(\CE, \CH) \in \K(\A_\qc(U))$ using \cite[Lemma 3.11]{gstck}, in view of the agreement between quasi-coherent sheaves and Cartesian presheaves. But $p^*\!\shom^\bullet_{\SX}(\CE, \CH) = \shom^\bullet_U(p^*\CE, p^*\!\CH)$, and this last complex is made of quasi-coherent sheaves because the components $p^*\CE$ are finitely generated locally free modules and those of $p^*\!\CH$ are quasi-coherent.
\end{proof}

\begin{prop} \label{perfsd}
Let $\SX$ be an Adams geometric stack. A perfect complex in $\D(\A_\qc(\SX))$ is strongly dualizable.
\end{prop}

\begin{proof}
Let $\CE$ be a perfect complex that we may suppose strictly perfect and $\CG \in \D(\A_\qc(\SX))$. Choose a $q$-injective resolution $\CG \to \CI_\CG$ in such a way that 
\[\rshom^\bullet_{\SX}(\CE, \CG) = \shom^\bullet_{\SX}(\CE, \CI_\CG).\] 
Being $\CE$ strictly perfect and $\CI_\CG$ quasi-coherent, it follows from Lemma \ref{perfqchom} that $\shom^\bullet_X(\CE, \CI_\CG)$ is quasi-coherent too. To show that 
\[
\dhom^\bullet_{\SX}(\CE, \CO_{\SX}) \otimes^{\LL}_{\CO_{\SX}} \CG \lto
\dhom^\bullet_{\SX}(\CE, \CG)
\]
is an isomorphism, we just have to prove that
\[
\alpha \colon
\rshom^\bullet_{\SX}(\CE, \CO_{\SX}) \otimes^{\LL}_{\CO_{\SX}} \CG \lto
\rshom^\bullet_{\SX}(\CE, \CG)
\]
is an isomorphism in $\D(\A_\qc(\SX))$. Take a presentation $p \colon U \to \SX$ and a $q$-injective resolution $\CO_{\SX} \to \CI_{\CO_{\SX}}$, we have

\begin{align*}
p^*(\rshom^\bullet_{\SX}(\CE, \CO_{\SX}) \otimes^{\LL}_{\CO_{\SX}} \CG) 
    & \cong p^*(\shom^\bullet_{\SX}(\CE, \CI_{\CO_{\SX}}) 
    \otimes^{\LL}_{\CO_{\SX}} \CG)\\   
    & \cong p^*\!\shom^\bullet_{\SX}(\CE, \CI_{\CO_{\SX}}) \otimes^{\LL}_{\CO_U} p^*\CG \\ 
    & \cong \shom^\bullet_U(p^*\CE, p^*\CI_{\CO_X}) \otimes^{\LL}_{\CO_U} p^*\CG  \\
    & \cong \shom^\bullet_U(p^*\CE, \CO_U) \otimes^{\LL}_{\CO_U} p^*\CG \tag{$\lozenge$}\\   
    & \cong \shom^\bullet_U(p^*\CE, p^*\CG)  \\   
    & \cong \shom^\bullet_U(p^*\CE, p^*\CI_{\CG})  \\
    & \cong p^*\!\shom^\bullet_{\SX}(\CE, \CI_{\CG})  \\   
    & \cong p^*\rshom^\bullet_{\SX}(\CE, \CG) 
\end{align*} 
The isomorphism ($\lozenge$) is a consequence that perfect complexes like $p^*\CE$ are strongly dualizable on an (affine) scheme. We conclude that $p^*\alpha$ is an isomorphism. Then, by faithful flatness of $p$, $\alpha$ is also an isomorphism.
\end{proof}

The following lemma is well-known, but we recall it here for the convenience of the readers. It expresses the agreement of two notions of generation of cocomplete triangulated categories, one of them used in \cite{hps}. Let $\A$ be a Grothendieck category. We say that $\D(\A)$ is generated by a subset of objects $\SS$ if $\D(\A)$ is its smallest triangulated subcategory stable for coproducts containing $\SS$.

\begin{lem}
$\D(\A)$ is generated by $\SS$ if and only if $\SS^\perp = 0$. 
\end{lem}

\begin{proof}
It follows, for instance,  from \cite[Theorem 5.7]{AJS} in view of \cite[Proposition 1.6]{AJS}.
\end{proof}

The next theorem gives a particular set of generators of $\D(\A_\qc(\SX))$. 

\begin{thm}\label{asht3} 
Let $\SX$ be an Adams geometric stack. Axiom (\ref{sdg}) of \cite[1.1]{hps} holds in the category $\D(\A_\qc(\SX))$. 
\end{thm}

\begin{proof}
In view of Proposition \ref{perfsd}, to prove that strongly dualizable objects generate $\D(\A_\qc(\SX))$, it is enough to prove that perfect complexes do. And this follows from the Adams condition. Indeed, by Theorem \ref{aisrp}, the category $\A_\qc(\SX)$ is generated in the sense of Abelian categories with coproducts by the set (of isomorphism classes) of locally free sheaves. Therefore the set of all of its suspensions, which are indeed perfect complexes themselves, generate $\D(\A_\qc(\SX))$ by the previous lemma.
\end{proof}

\begin{rem}
The previous result does not imply that $\D(\A_\qc(\SX))$ is compactly generated. By \cite[Theorem 1.3]{HNR} there are examples of geometric stacks $\SX$ in which $\D_\qc(\SX)$ and therefore $\D(\A_\qc(\SX))$ are not compactly generated. Notice that $\SB\mathbb{G}_a$, the classifying stack of the additive group $\mathbb{G}_a$, is an Adams geometric stack that satisfies the hypothesis in \lc
\end{rem}

\begin{cor}\label{st=sht}
 Let $\SX$ be an Adams geometric stack, the category $\D(\A_\qc(\SX))$ is a stable homotopy category in the sense of \cite{hps}.
\end{cor}

\begin{proof}
Combine Theorems \ref{asht145}, \ref{asht2}, \ref{asht3}.
\end{proof}

\begin{rem}
 In a recent preprint Hall and Rydh prove that for a geometric stack $\SX$ with \emph{quasi-finite} diagonal, the category $\D_\qc(\SX)$ is compactly generated by a single perfect complex \cite[Theorem A]{hr}. Thus, in this case, the category $\D(\A_\qc(\SX))$ is an \emph{algebraic} stable homotopy category, in the terminology of \cite[Definition 1.1.4]{hps}.
\end{rem}

\section{The closed structure and comodules}

\begin{cosa}
Let $A_\bullet$ be a Hopf algebroid and $\SX = \stack(A_\bullet)$ its associated geometric stack. Let us relate the comodule tensor product over $A_\bullet$ with the just described tensor product in $\A_\qc(\SX)$.

Let $(M, \psi_M)$ and $(N, \psi_N)$ be left $A_\bullet$-comodules. It is possible to define a structure of comodule on $M \otimes_{A_0} N$, that we will denote by $M \otimes_{A_\bullet}^\cc N$ and whose structure map is given by the composition
\[
M \otimes_{A_0} N 
\xto{\psi_M \otimes \psi_N} (A_1\: {_{\eta_R} \otimes_{A_0}}M) \otimes_{A_0} (A_1\: {_{\eta_R} \otimes_{A_0}} N) 
\overset{g}{\lto} A_1\: {_{\eta_R} \otimes_{A_0}}M \otimes_{A_0} N
\]
where $g(a \otimes m \otimes a' \otimes n) = aa' \otimes m \otimes n$. We follow the definition given by Hovey in \cite[Lemma 1.1.2]{hov} (where $\otimes^\cc$ is denoted by $\wedge$).
\end{cosa}

Denote by $A_{\bullet}\com$ the category of $A_\bullet$-comodules. Recall the equivalence
\[
\GGamma_p^{\SX} \colon \A_\qc(\SX) \liso A_{\bullet}\com
\]
from \cite[Corollary 5.9]{gstck}. Our next task will be to show that his equivalence respects the closed structure on the corresponding categories.

\begin{prop}\label{comptor}
We have that
\[
\GGamma_p^{\SX}(\CF \otimes_{\CO_{\SX}} \CG) \liso 
\GGamma_p^{\SX}(\CF) \otimes_{A_\bullet}^\cc \GGamma_p^{\SX}(\CG)
\]
\end{prop}

\begin{proof}
Let us recall some basic notation. We have a Cartesian diagram
\begin{equation}\label{phi}
\begin{tikzpicture}[baseline=(current  bounding  box.center)]
\matrix(m)[matrix of math nodes, row sep=2.6em, column sep=2.8em,
text height=1.5ex, text depth=0.25ex]{
  U\times_{\SX}U & U \\
  U              & \SX \\};
\path[->,font=\scriptsize,>=angle 90] (m-1-1) edge
node[auto] {$p_2$}
(m-1-2) edge node[left] {$p_1$} (m-2-1)
(m-1-2) edge node[auto] {$p$} (m-2-2)
(m-2-1) edge node[auto] {$p$} (m-2-2);
\draw [shorten >=0.2cm,shorten <=0.2cm, ->, double] (m-2-1) -- (m-1-2) node[auto, midway,font=\scriptsize]{$\phi$};
\end{tikzpicture}
\end{equation}
Note that $U = \spec(A_0)$ and $U \times_{\SX} U = \spec(A_1)$. Let $M = \CF(U,p)$ and $N = \CG(U,p)$. Then $\GGamma_p^{\SX}(\CF) = (M, \psi_M)$ and $\GGamma_p^{\SX}(\CG) = (N, \psi_N)$, where
\[
\psi_M = (\wadj{\CF(p_2,\iid)})^{-1} \CF(p_1,\phi^{-1}) \quad 
\text{ and } 
\quad \psi_N = (\wadj{\CG(p_2,\iid)})^{-1} \CG(p_1,\phi^{-1})
\]
We have to compare $g \circ (\psi_M \otimes \psi_N)$ with the structure map of $P :=\GGamma_p^{\SX}(\CF \otimes_{\CO_{\SX}} \CG)$. As before, this map is
\[
\psi_{P} =  (\wadj{((\CF \otimes_{\CO_{\SX}} \CG)(p_2,\iid))})^{-1} 
(\CF \otimes_{\CO_{\SX}} \CG)(p_1,\phi^{-1})
\]
We have a commutative diagram, with $U_1 := U \times_{\SX} U$,
\[
\begin{tikzpicture}
      \draw[white] (0cm,2cm) -- +(0: \linewidth)
      node (G) [black, pos = 0.22] {$A_1 \otimes_{A_0} (\CF(U,p) \otimes_{A_0} \CG(U,p))$}
      node (H) [black, pos = 0.78] {$\CF(U_1,pp_2) \otimes_{A_1} \CG(U_1,pp_2)$};
      \draw[white] (0cm,0.25cm) -- +(0: \linewidth)
      node (E) [black, pos = 0.5] {$(A_1 \otimes_{A_0} (\CF(U,p)) \otimes_{A_1} (A_1 \otimes_{A_0} \CG(U,p))$};
      \draw [->] (G) -- (H) node[above, midway, scale=0.75]{$\wadj{((\CF \otimes_{\CO_{\SX}} \CG)(p_2,\iid))}$};
      \draw [->] (G) -- (E) node[auto, swap, midway, scale=0.75]{$w$};
      \draw [->] (E) -- (H) node[auto, swap, midway, scale=0.75]{$\wadj{\CF(p_2,\iid)}\otimes_{A_1}\wadj{\CG(p_2,\iid)}$};  
\end{tikzpicture}
\]
notation as in Proposition \ref{qcten}. Therefore,
\begin{align*}
\psi_{P} &= w^{-1} \circ (\wadj{\CF(p_2,\iid))} \otimes_{A_1} \wadj{\CG(p_2,\iid)})^{-1} \circ (\CF \otimes_{\CO_{\SX}} \CG)(p_1,\phi^{-1}) \\ 
    & = w^{-1} \circ (\wadj{\CF(p_2,\iid)} \otimes_{A_1} \wadj{\CG(p_2,\iid)})^{-1} \circ (\CF(p_1,\phi^{-1}) \otimes_{A_1} \CG(p_1,\phi^{-1})  \\
    & = g \circ         
    (\wadj{(\CF(p_2,\iid)})^{-1} \circ \CF(p_1,\phi^{-1}))
    \otimes_{A_0} 
    (\wadj{(\CG(p_2,\iid)})^{-1} \circ \CG(p_1,\phi^{-1}))
    \\ 
    & = g \circ (\psi_M \otimes \psi_N)             
\end{align*}
as wanted.
\end{proof}

\begin{cosa} \textbf{The \emph{hom} comodule.}
Let, as before, $A_\bullet$ be a Hopf algebroid and $\SX = \stack(A_\bullet)$ its associated geometric stack. Given $A_\bullet$-comodules $M$, $N$ there is an associated abelian group $\Hom_{A_\bullet\com}(M, N)$. This does not support the comodule structure for an internal hom in the category, as one can realize by considering the case in which $M = A_0$. See the example after Corollary \ref{comphom}. The construction of the comodule that corresponds to the internal hom will start with a simple characterization when the second comodule $N$ is extended and from this case one defines it for every comodule. The following easy observation will be of use.

\begin{lem}\label{adjmoco}
There is an adjunction
\[
\begin{tikzpicture}
\matrix (m) [matrix of math nodes, row sep=3em, column sep=3em]{
A_\bullet\com &  A_0\md \\
}; 
\draw [transform canvas={yshift= 0.3ex},font=\scriptsize,<-]
(m-1-1) -- node[above]{$A_1\otimes_{A_0}-$}(m-1-2);
\draw [transform canvas={yshift= -0.3ex},font=\scriptsize,->]
(m-1-1) -- node[below]{$\SU$}(m-1-2);
\end{tikzpicture}
\]
where for $M \in A_\bullet\com$ we define $\SU(M, \psi) = M$ and for $N \in A_0\md$ the module $A_1 \otimes_{A_0} N$ is endowed with its canonical extended comodule structure.
\end{lem}

\begin{proof}
It is a special case of \cite[7.3]{gstck}.
\end{proof}

From now on we will mostly omit the structure morphism in $(M, \psi)$ and we will use systematically the notation $\psi = \psi_M$ for it. To our end, first consider $M \in A_{\bullet}\com$ and $N' \in A_0\md$. Define
\[
\chom{A_\bullet}(M, A_1 \otimes_{A_0} N') := 
A_1 \otimes_{A_0} \Hom_{A_0}(M, N')
\]
with its structure of extended comodule, see \cite[5.3]{gstck}.

Let $\varphi \colon A_1 \otimes_{A_0} N \to A_1 \otimes_{A_0} N'$ be a homomorphism of extended comodules. Let us describe
\(
\chom{A_\bullet}(M, \varphi) \colon 
\chom{A_\bullet}(M, A_1 \otimes_{A_0} N) \to 
\chom{A_\bullet}(M, A_1 \otimes_{A_0} N')
\)
explicitly as follows. Consider the composition 
\[
(\varepsilon \otimes \id)\varphi(\id \otimes \ev) \colon
A_1 \otimes_{A_0} \Hom_{A_0}(M, N) \otimes_{A_0} M \lto N'
\]
where $\ev \colon \Hom_{A_0}(M, N) \otimes_{A_0} M  \to N$ denotes the evaluation morphism and $\varepsilon \colon A_1 \to A_0$ the counit of $A_{\bullet}$.
By adjunction one obtains an $A_0$-linear map
\[
\overline{\varphi} \colon A_1 \otimes_{A_0} \Hom_{A_0}(M, N) \lto
\Hom_{A_0}(M, N') 
\]
and we apply the adjunction in Lemma \ref{adjmoco} which yields the desired comodule homomorphism $\chom{A_\bullet}(M, \varphi)$. Further, this construction induces a natural isomorphism for extended comodule homomorphisms
\begin{equation}\label{adjind}
\Hom_{A_\bullet\com}(P, \chom{A_\bullet}(M, A_1 \otimes_{A_0} N))) \iso
\Hom_{A_\bullet\com}(P \otimes_{A_\bullet}^\cc M, A_1 \otimes_{A_0} N)).
\end{equation}

Now, present a general comodule $N \in A_{\bullet}\com$ as a kernel of extended comodules
\begin{equation}\label{prescomod}
 0 \lto N \xto{\psi_N} A_1 \otimes_{A_0} N \lto A_1 \otimes_{A_0} N'
\end{equation}
and define the \emph{hom} comodule as the following kernel
\[
0 \lto \chom{A_\bullet}(M, N) 
\lto \chom{A_\bullet}(M, A_1 \otimes_{A_0} N) 
\lto \chom{A_\bullet}(M, A_1 \otimes_{A_0} N').
\]
This definition extends to all comodules by functoriality of kernels in view of the presentation (\ref{prescomod}).
\end{cosa}

\begin{prop}\label{adjexttorcom}
Let $A_\bullet$ be a Hopf algebroid and $P, M, N \in A_{\bullet}\com$. We have the following isomorphism
\[
\Hom_{A_\bullet\com}(P, \chom{A_\bullet}(M, N)) \iso
\Hom_{A_\bullet\com}(P \otimes_{A_\bullet}^\cc M, N).
\]
This establishes an adjunction $-\otimes_{A_\bullet}^\cc M \dashv \chom{A_\bullet}(M, -)$ in $A_{\bullet}\com$.
\end{prop}

\begin{proof}
Let $f \colon P \otimes_{A_\bullet}^\cc M \to N$ and consider the composition
\[
P \otimes_{A_\bullet}^\cc M \overset{f}{\lto} N \overset{\psi_N}{\lto} A_1 \otimes_{A_0} N
\xto{\varphi} A_1 \otimes_{A_0} N'
\]
where $\varphi$ is induced by a presentation like (\ref{prescomod}). The morphism $\psi_N \circ f$ corresponds to $\overline{f} \colon P \to \chom{A_\bullet}(M, A_1 \otimes_{A_0} N)$ by \eqref{adjind}. Notice that the  morphism $\chom{A_\bullet}(M, \varphi) \circ \overline{f} = 0$, therefore $\overline{f}$ factors through a map
\[
\wadj{f} \colon P \lto \chom{A_\bullet}(M, N).
\]

Conversely, let $g \colon P \to \chom{A_\bullet}(M, N)$. It induces a comodule morphism $g' \colon P \to \chom{A_\bullet}(M, A_1 \otimes_{A_0} N)$  that corresponds by invoking again \eqref{adjind} to $\overline{g} \colon P \otimes_{A_\bullet}^\cc M \to A_1 \otimes_{A_0} N$. It holds that $\varphi \circ \overline{g} = 0$ and this gives a morphism
\[
\wadj{g} \colon P \otimes_{A_\bullet}^\cc M \lto N
\]
One checks that these two maps correspond to the desired adjunction.
\end{proof}

\begin{cor}
Let $A_\bullet$ be a Hopf algebroid and $P, M, N \in A_{\bullet}\com$. The previous adjunction also hold internally, \ie{} we have the following isomorphism
\[
\chom{A_\bullet}(P, \chom{A_\bullet}(M, N)) \liso
\chom{A_\bullet}(P \otimes_{A_\bullet}^\cc M, N).
\]
\end{cor}

\begin{proof}
This follows as usual formally from the axioms of closed category, see \cite[Exercise (3.5.3)(e)]{yellow}.\end{proof}

\begin{rem}
 As a consequence we have the following isomorphism
\[
\Hom_{A_\bullet\com}(A_0, \chom{A_\bullet}(M, N)) \liso
\Hom_{A_\bullet\com}(M, N).
\]
expressing the relationship between external and internal homs.
\end{rem}

\begin{cor}\label{comphom}
 For $\CF$ and $\CG$ in $\A_\qc(\SX) $, we have the following isomorphism of bifunctors
 \[
\GGamma_p^{\SX}(\SQ_{\SX}\!\shom_{\SX}(\CF, \CG)) \liso 
\chom{A_\bullet}(\GGamma_p^{\SX}(\CF), \GGamma_p^{\SX}(\CG))
\]
\end{cor}

\begin{proof}
 Equivalences of categories transform adjunctions into adjunctions, so the isomorphism follows from \cite[Corollary 5.9]{gstck}, using Proposition \ref{comptor} and Proposition \ref{adjexttorcom}.
\end{proof}

\begin{ex}
Let $G = \spec(H)$ be an affine flat algebraic group over an affine scheme $\spec(R)$. Assume that $G$ acts on the affine scheme $X = \spec(A)$. There is a structure of Hopf algebra on $H$ \cite[I. 2.3]{rag}. We are in the setting of \cite[Examples 5.2 and 5.11]{gstck}. In this situation, the pair $(A, H \otimes_{R} A)$ constitutes a Hopf algebroid with flat structure maps. 

Consider the stack:
\[
[X/G] := \stack(A, H \otimes_{R} A).
\]
This stack is the geometric quotient of the scheme $X$ by the action of $G$. The canonical quotient map $p \colon X \to [X/G]$ is a presentation. 

A comodule $M$ over $(A, H \otimes_{R} A)$ corresponds to a $G$-equivariant quasi-coherent sheaf over $X$. According to \cite[I. 2.10 (5)]{rag} (adapting the notation) for two comodules $M$ and $N$ we have that
\[
\Hom_{(A, H \otimes_{R} A)\com}(M, N) \equiv
\Hom_{G}(M, N) =
\Hom_{R}(M, N)^G
\]
while, for $M = R$ we have
\[
\chom{(A, H \otimes_{R} A)}(R, N) = N.
\]
In general, $N$ differs from $N^G$ unless the action is trivial.
\end{ex}

\begin{cosa} \textbf{The derived setting.}
Let $A_\bullet$ be a Hopf algebroid and $\SX = \stack(A_\bullet)$. Having concluded that $\GGamma_p^{\SX}$ (and, as a consequence its quasi-inverse) respects the closed structure on the equivalent categories, we transport this fact to the setting of derived categories. We denote by $\D(A_{\bullet}) := \D(A_{\bullet}\com)$ the derived category of complexes of $A_{\bullet}$-comodules. Being $\GGamma_p^{\SX}$ exact, it induces an equivalence of derived categories
\[
\GGamma_p^{\SX} \colon \D(\A_\qc(\SX)) \liso \D(A_{\bullet})
\]
that, as it is customary, we keep denoting the same. Notice that the bifunctor $\chom{A_\bullet}(-, -)$ in $A_\bullet\com$ defines a bifunctor \emph{hom complex} in the category $\CCC(A_{\bullet}) := \CCC(A_{\bullet}\com)$ of complexes of $A_{\bullet}$-comodules, compare \cite[(1.5.3)]{yellow}. Let $A_{\bullet}$ be an Adams Hopf algebroid, \ie{} a Hopf algebroid such that $\stack(A_\bullet)$ is an Adams geometric stack. In this case, imitating the proof of Lemma \ref{lema4}, one sees that it induces a bifunctor in $\D(A_{\bullet})$ that we write $\cdhom{A_\bullet}(-, -)$. Analogously, the derived functor of 
$- \otimes_{A_\bullet}^\cc - $ is defined and denoted $- \otimes_{A_\bullet}^{\cc\,\LL} -$. The next result upgrades the previous discussion to the derived setting.
\end{cosa}

\begin{prop}\label{compder}
For $\CF$ and $\CG$ in $\D(\A_\qc(\SX))$, we have the following isomorphism of bifunctors taking values in $\D(A_{\bullet})$
\[
\GGamma_p^{\SX}(\CF \otimes^{\LL}_{\CO_{\SX}} \CG) \liso 
\GGamma_p^{\SX}(\CF) \otimes_{A_\bullet}^{\cc\,\LL} \GGamma_p^{\SX}(\CG)
\]
\[
\GGamma_p^{\SX}(\dhom^\bullet_{\SX}(\CF, \CG)) \liso 
\cdhom{A_\bullet}(\GGamma_p^{\SX}(\CF), \GGamma_p^{\SX}(\CG))
\] 
\end{prop}

\begin{proof}
By the previous discussion, they are consequences of Propostion \ref{comptor} and Corollary \ref{comphom}, respectively.
\end{proof} 

\begin{rem}
The equivalence between $\D(A_{\bullet})$ and $\D(\A_\qc(\SX))$, interchanges their corresponding symmetric closed monoidal structures. Moreover, it is clear that the collection of bounded complexes in $\D(A_{\bullet})$ whose underlying $A_0$-modules are projective and finitely generated are strongly dualizable objects and, by virtue of Theorem \ref{asht3}, they generate the derived category.
\end{rem}

\begin{cor}
Let $A_{\bullet}$ be an Adams Hopf algebroid. Then, the category $\D(A_\bullet)$ is a stable homotopy category in the sense of \cite{hps}. 
\end{cor}

\begin{proof}
Combine Corollary \ref{st=sht} with Proposition \ref{compder} in view of the previous discussion.
\end{proof}

\begin{cosa}\label{compHov}
\textbf{Comparison with Hovey's results}.
 Hovey, in \cite[\S 2]{hov}, defines the projective model structure on $\CCC(A_{\bullet})$. For our purposes it is enough to specify the class of weak equivalences. Let $\phi \colon M \to N$ denote a homomorphism in $\CCC(A_{\bullet})$. The class of weak equivalences is defined by
 \[
 \SW_{\hv} := \{ \phi \, /\, \Hom_{A_\bullet\com}(P, \phi) \text{ is a quasi-isomorphism } \forall P \in \SP \}
 \]
 where $\SP$ denotes the class of comodules whose underlying $A_0$-module is projective. He works systematically with the corresponding homotopy category. Let us denote it by
 \[
 \D_{\hv}(A_{\bullet}) := \CCC(A_{\bullet})[\SW_{\hv}^{-1}]
 \]
 Let $\SW$ denote the class of all quasi-isomorphisms in $\CCC(A_{\bullet})$; by \cite[Proposition 2.1.5]{hov}, we have $\SW_{\hv} \subset \SW$. This implies the existence of a canonical $\Delta$-functor
\[
\D_{\hv}(A_{\bullet}) \lto \D(A_{\bullet}).
\]
Notice that our constructions are the usual ones from homological algebra without fixing a particular model structure on $\CCC(A_{\bullet})$. We ignore if this functor is monoidal. Moreover, we do not know how Hovey's model structure on $\CCC(A_{\bullet})$ compares with those that have quasi-isomorphisms as weak equivalences, like the injective structure in \cite{hmc} or the flat one in \cite{gill}, among others. 
\end{cosa}



\end{document}